\documentclass[11pt,a4paper]{amsart}
\usepackage{tikz}
\usepackage{amssymb}
\usepackage{enumitem}

\newtheorem{theorem}{Theorem}[section]
\newtheorem{lemma}[theorem]{Lemma}

\newtheorem{proposition}[theorem]{Proposition}

\newcommand{\clone}[1]{\overline{#1}}
\newcommand{\fig}[1]{\includegraphics[page=#1]{edge-crit-figs}}
\newcommand{\hf}{\hspace*{0mm}\hfill\hspace*{0mm}}

\newenvironment{xcase}[2]{\smallskip\par\noindent\emph{Case #1:} \emph{#2}.}{\par\smallskip}

\tikzstyle{vertex}=[circle, fill=black, minimum size=6pt, inner sep=0]
\tikzstyle{edge}=[semithick]

\begin{document}\suppressfloats
\title{Edge-critical subgraphs of Schrijver graphs}
\author{Tom\'a\v s Kaiser}
\thanks{The first author was supported by project GA17-04611S of the Czech Science
  Foundation.}
\address{Department of Mathematics and European Centre of Excellence
  NTIS (New Technologies for the Information Society), University of
  West Bohemia, Pilsen, Czech Republic.}
  \email{kaisert@kma.zcu.cz}
\author{Mat\v ej Stehl\'ik}
\thanks{The second author was partially supported by ANR project GATO
(ANR-16-CE40-0009-01) and by LabEx PERSYVAL-Lab (ANR-11-LABX-0025).}
\address{Laboratoire G-SCOP, Univ.\ Grenoble Alpes, France}
\email{matej.stehlik@grenoble-inp.fr}

\begin{abstract}
  For $k\geq 1$ and $n\geq 2k$, the Kneser graph $KG(n,k)$ has all
  $k$-element subsets of an $n$-element set as vertices; two such
  subsets are adjacent if they are disjoint. It was first proved by
  Lov\'{a}sz that the chromatic number of $KG(n,k)$ is
  $n-2k+2$. Schrijver constructed a vertex-critical subgraph $SG(n,k)$
  of $KG(n,k)$ with the same chromatic number. For the stronger notion
  of criticality defined in terms of removing edges, however, no
  analogous construction is known except in trivial cases. We provide
  such a construction for $k=2$ and arbitrary $n\geq 4$ by means of a
  nice explicit combinatorial definition.
\end{abstract}
\maketitle

\section{Introduction}

For positive integers $n,k$, where $n\geq 2k$, the \emph{Kneser graph}
$KG(n,k)$ has all $k$-element subsets of the set $[n] = \{1,\dots,n\}$
as its vertices, with edges joining disjoint pairs of subsets. It was
conjectured by Kneser~\cite{K55} and proved by Lov\'{a}sz~\cite{L78}
that the chromatic number of $KG(n,k)$ is
$n-2k+2$. Schrijver~\cite{S78} proved that there is a subgraph of
$KG(n,k)$ that is in general much smaller and still has chromatic
number $n-2k+2$. This is the \emph{Schrijver graph} $SG(n,k)$, defined
as the induced subgraph of $KG(n,k)$ on the set of all \emph{stable}
$k$-subsets of $[n]$ --- that is, those that contain no pair of
consecutive elements nor the pair $1,n$. In fact, Schrijver proved
that $SG(n,k)$ is \emph{vertex-critical}, i.e., the removal of any
vertex of $SG(n,k)$ decreases the chromatic number.

It is natural to ask whether $SG(n,k)$ satisfies the stronger
condition of criticality defined in terms of removing edges. A graph
$G$ is said to be \emph{edge-critical} if $\chi(G-e) < \chi(G)$ for
each edge $e$ of $G$, where $\chi$ denotes the chromatic
number. Equivalently, $G$ is edge-critical if none of its proper
subgraphs has the same chromatic number. For instance, the graph
$SG(2k+1,k)$ is edge-critical (being isomorphic to a cycle of length
$2k+1$) and so is $SG(n,1)$ (the complete graph $K_n$), but this is
not the case for $SG(n,k)$ with $n\geq 2k+2$ and $k \geq 2$.

In this paper, we give a simple combinatorial description of an
edge-critical spanning subgraph of the graph $SG(n,2)$ (for any
$n\geq 4$) that was discovered in the course of our work on colouring
quadrangulations of projective spaces~\cite{KS15,KS17}. This is the
first step to a description of such edge-critical subgraphs in
$SG(n,k)$ for general $k$, which is currently work in progress.

For every integer $n \geq 4$, we define the graph $G_n$ as follows.
The vertex set of $G_n$ is the set of all stable $2$-subsets of
$[n]$. A stable subset $\{a,b\}$, where $a<b$, is denoted by $ab$.
Edges in $G_n$ only join disjoint pairs of $2$-subsets of $[n]$. Let
$ab$ and $cd$ be such a pair, where $a < c$. The vertices $ab$ and
$cd$ are adjacent in $G_n$ if and only if one of the following holds:
\begin{itemize}
\item $a < c < b < d$ (a \emph{crossing} pair), or
\item $1 < a < c < d < b$ (a \emph{transverse} pair).
\end{itemize}
See Figure~\ref{fig:pairs} for an illustration of the definition. In
this figure, vertices of $G_n$ are visualised as chords of the cycle
$C_n$. Accordingly, we sometimes refer to vertices of $G_n$ as
\emph{chords} of $C_n$.

\begin{center}
  \begin{figure}
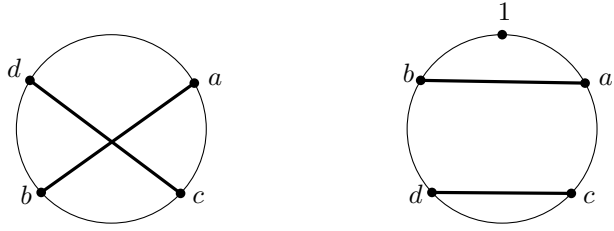

    \hf\fig1\hf\fig2\hf
    \caption{A pair of chords $ab,cd$ that is crossing (left) and
      another that is transverse (right).}
    \label{fig:pairs}
  \end{figure}
\end{center}

The main result of this paper is the following.

\begin{theorem}
\label{thm:main}
For every $n \geq 4$, the graph $G_n$ is $(n-2)$-chromatic and
edge-critical.
\end{theorem}

Since $G_n$ is an edge-critical spanning subgraph of $SG(n,2)$, it is
interesting to compare the number of edges in $G_n$ to the number of
edges in $SG(n,2)$.  The only case where $G_n$ is not a proper
subgraph of $SG(n,2)$ is when $n \leq 5$: $G_4$ and $SG(4,2)$ are both
isomorphic to $K_2$, while $G_5$ and $SG(5,2)$ are both isomorphic to
$C_5$. For $n>5$, $G_n$ is a proper subgraph of $SG(n,2)$, and the
following proposition determines the asymptotic ratio of their sizes.

\begin{proposition}\label{p:ratio}
  As $n \to \infty$, the ratio $|E(G_n)|/|E(SG(n,2))|$ tends to $2/3$.
\end{proposition}
\begin{proof}
  Each edge of $G_n$ corresponds to either a crossing pair or a
  transverse pair of chords of the cycle $C_n$. Let us call the pairs
  of disjoint chords of $C_n$ not corresponding to any of these types
  \emph{lateral}. (That is, chords $ab$ and $cd$ form a lateral pair
  if $a < b < c < d$ or $c < d < a < b$.)

  Let us estimate the number of pairs of each of these three
  types. Any pair of chords determines a $4$-tuple of elements of
  $[n]$, namely the endvertices of the chords. For crossing pairs,
  this is in fact a $1$--$1$ correspondence, since any $4$-element
  subset of $[n]$ determines precisely one crossing pair. Thus, the
  number of crossing pairs is $\binom n 4$.

  For transverse and lateral pairs, the correspondence is no longer
  one-to-one, but it is not hard to show that the number of pairs of
  each of these types is $\binom n 4 - O(n^3)$. It follows that
  \begin{align*}
    |E(G_n)| &= 2\binom n 4 - O(n^3),\\
    |E(SG(n,2))| &= 3\binom n 4 - O(n^3),
  \end{align*}
  so the asymptotic ratio of these two quantities is $2/3$ as claimed.
\end{proof}


\section{Proof of Theorem~\ref{thm:main}}

The Mycielski construction~\cite{M55} is one of the earliest and
arguably simplest constructions of triangle-free graphs of arbitrarily
high chromatic number. Given a graph $G=(V,E)$, we let $M(G)$ be the
graph with vertex set $V \cup \{\clone u:\,u\in V\} \cup \{*\}$,
where there are edges $\{u,v\}$ and $\{u,\clone v\}$ whenever
$\{u,v\} \in E$, and an edge $\{\clone u,*\}$ for all $u \in
V$. For each $u\in V$, the vertex $\clone u$ is referred to as the
\emph{clone} of $u$ in $M(G)$.

It is an easy exercise to show that the chromatic number increases
with each iteration of $M(\cdot)$. Let $M_k$ be the graph obtained
from $K_2$ by iterating the Mycielski construction $k-2$ times.
It is easy to see that $M_k$ is $k$-chromatic (in fact, $M_k$ is
$k$-edge-critical).

Theorem~\ref{thm:main} follows immediately from the following two lemmas.

\begin{lemma}
  \label{lem:homomorphism}
  For $n\geq 5$, there is a homomorphism
  \[
  h:\,M(G_{n-1}) \to G_n.
  \]
  In particular, $\chi(G_n)\geq n-2$.
\end{lemma}
\begin{proof}
  Let $ab$ be a vertex of $G_{n-1}$; recall that this means $a <
  b$. The clone of $ab$ in $M(G_{n-1})$ is denoted by $\clone{ab}$. We
  define
  \begin{align*}
    h(ab) &= ab,\\
    h(\clone{ab}) &= \begin{cases}
      an & \text{if $a\neq 1$,}\\
      bn & \text{if $a = 1$,}
    \end{cases}\\
    h(*) &= \{1,n-1\}.
  \end{align*}

  Observe first that in all cases, the value of the mapping $h$ is a
  vertex of $G_n$. The only case that needs an explanation is that of
  $h(\clone{ab})$. Here, if $a\neq 1$, then $an$ is stable, since $a
  < b \leq n-1$. On the other hand, if $a=1$, then $2 < b < n-1$ (since $ab$ is a
  vertex of $G_{n-1}$), so $bn$ is stable.

  Let us show that $h$ is a homomorphism. We consider an edge $e$ of
  $M(G_{n-1})$ and prove, in each of the following cases, that the
  image of $e$ under $h$ is an edge of $G_n$.

  \begin{xcase}{1}{$e$ is an edge of $G_{n-1}$}
    Note that $h(ab)=ab$ and $h(cd)=cd$. Since a crossing pair of
    vertices of $G_{n-1}$ is also crossing in $G_n$, and similarly for
    a transverse pair, $e$ is an edge of $G_n$.
  \end{xcase}
  \begin{xcase}{2}{$e$ has endvertices $ab$ and $\overline{cd}$, where
      $ab$ and $cd$ form an edge of $G_{n-1}$}
    We have $h(ab)=ab$. For $h(\clone{cd})$, we have
    \begin{equation*}
      h(\clone{cd}) =
      \begin{cases}
        cn & \text{if $c > 1$,}\\
        dn & \text{if $c = 1$.}
      \end{cases}
    \end{equation*}
    
    Suppose first that the pair $ab,cd$ is crossing in
    $G_{n-1}$. (This subcase is illustrated in
    Figure~\ref{fig:cases}.) If $a < c < b < d$, then
    $h(\clone{cd}) = cn$ and $a < c < b < n$, so the pair $ab,cn$ is
    crossing in $G_n$. (In particular, it is disjoint, which will not
    be repeated in the following subcases.)  If $1 < c < a < d < b$,
    then again $h(\clone{cd}) = cn$ and $c < a < b < n$, so the pair
    $ab,cn$ is transverse in $G_n$. Finally, if $c=1$ (and
    $1 < a < d < b$), then $h(\clone{cd}) = dn$ and $a < d < b < n$,
    so the pair $ab,dn$ is crossing.

    Suppose then that $ab,cd$ is a transverse pair in $G_n$. Thus,
    $1\notin\{a,b,c,d\}$ and in particular $h(\clone{cd}) = cn$. If $a
    < c < d < b$, then $a < c < b < n$ and the pair $ab,cn$ is
    crossing. If $c < a < b < d$, then $c < a < b < n$ and $ab,cn$ is
    a transverse pair.
  \end{xcase}
  \begin{xcase}{3}{$e$ has endvertices $\clone{cd}$ and $*$, where
      $cd$ is a vertex of $G_{n-1}$}
    We have $h(*) = \{1,n-1\}$. Since $n\in h(\clone{cd})$, the pair
    $h(\clone{cd}),h(*)$ must be crossing.
  \end{xcase}
  This concludes the proof that $h$ is a homomorphism. It follows that
  $\chi(G_n)\geq \chi(M(G_{n-1})$. The statement that
  $\chi(G_n)\geq n-2$ follows by induction with base case
  $\chi(G_5) = 3$: if we know that $\chi(G_{n-1})\geq n-3$, then
  $\chi(M(G_{n-1})) \geq n-2$ since --- as remarked above --- an application
  of $M(\cdot)$ increases the chromatic number, and therefore
  $\chi(G_n)\geq n-2$.
\end{proof}

\begin{center}
  \begin{figure}
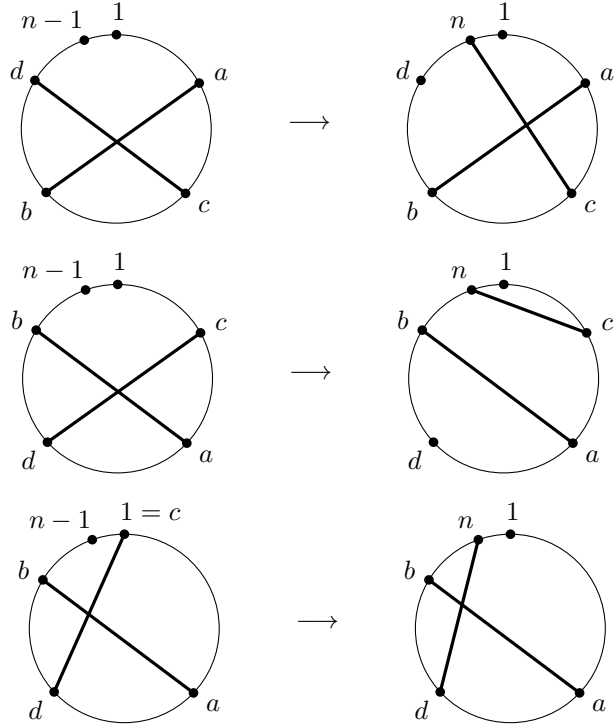

    \fig3\\[3mm]\fig4\\[3mm]\fig5
    \caption{An illustration to Case 2 in the proof of
      Lemma~\ref{lem:homomorphism}. Left: a pair of chords $ab$, $cd$
      in $C_{n-1}$. Right: the images of $ab$ and $\clone{cd}$ under
      $h$. Note how crossing pairs are transformed to either crossing
      or transverse pairs.}
    \label{fig:cases}
  \end{figure}
\end{center}

\begin{lemma}
\label{lem:critical}
  For every $n \geq 4$ and every edge $e \in E(G_n)$, the graph $G_n-e$ is
  $(n-3)$-colourable.
\end{lemma}

\begin{proof}
  Let $e$ be an edge of $G_n$ with endvertices $ab$ and $cd$. We will
  show that $G-e$ is $(n-3)$-colourable.

  \begin{xcase}{1}{$ab$ and $cd$ are a crossing pair and
      $1\in\{a,b,c,d\}$}
    Without loss of generality, assume that $a=1$ and $1 < c < b <
    d$. Let $A=\{a,b,c,d\}$. We first colour all stable $2$-subsets of
    $[n]$ not contained in $A$ using the $n-4$ colours from the set
    $[n]\setminus A$: any such stable $2$-subset $xy$ is coloured by
    $\min(\{x,y\}\setminus\{a,b,c,d\})$. Observe that this partial
    colouring is proper in $G_n$. (We will call it the \emph{min-based
      colouring with respect to $A$}.)

    Having used $n-4$ colours, we have one colour left for the stable
    $2$-subsets of $A$. Since $a=1$, no pair of these subsets is
    transverse. Additionally, there is only one crossing pair, namely
    $ab$ and $cd$. Assign a new colour $\ell_1$ to each stable
    $2$-subset of $A$; we obtain a proper $(n-3)$-colouring of
    $G_n-e$. (See Figure~\ref{fig:chords} for an illustration of this
    and the following cases.)
  \end{xcase}

  \begin{xcase}{2}{$ab$ and $cd$ are a crossing pair and
      $1\notin\{a,b,c,d\}$}
    Let $A=\{1,a,b,c,d\}$ and start with the min-based colouring with
    respect to $A$. This uses $n-5$ colours.
    
    We will colour stable $2$-subsets of $A$ using two new colours
    $\ell_1$ and $\ell_2$. Without loss of generality, assume that
    $a < c < b < d$. Assigning colour $\ell_1$ to all the stable
    $2$-subsets in the set $\{1a, 1b, 1c, 1d, bc, bd\}$ and colour
    $\ell_2$ to those in $\{ab,ac,ad,cd\}$, it is easy to check that
    the $2$-colouring of the induced subgraph of $G_n-e$ on the set of
    stable $2$-subsets of $A$ is proper. Consequently, we obtain a
    proper $(n-3)$-colouring of $G_n-e$.
  \end{xcase}
  
  \begin{xcase}{3}{$ab$ and $cd$ are a transverse pair}
    By the assumption, $1\notin\{a,b,c,d\}$. Without loss of
    generality, assume that $1 < a < c < d < b$. Since $cd$ is stable,
    there is $x\in [n]$ such that $c < x < d$. Let $A=\{1,a,b,c,d,x\}$
    and start with a min-based colouring with respect to $A$ using
    $n-6$ colours.

    To colour the uncoloured stable $2$-subsets of $[n]$ using new
    colours $\ell_1$, $\ell_2$ and $\ell_3$, we use the
    following rule:
    \begin{itemize}
    \item $\ell_1$ is assigned to stable $2$-subsets in $\{1a,1x,1d,ax,dx\}$, 
    \item $\ell_2$ is assigned to those in $\{1b,1c,bc,bx,cx\}$, 
    \item $\ell_3$ is assigned to those in $\{ab,ac,ad,cd,bd\}$. 
    \end{itemize}
    We obtain a proper $(n-3)$-colouring of $G_n-e$ since no crossing
    pair gets the same colour, and the only monochromatic transverse
    pair is $ab,cd$.
  \end{xcase}
\end{proof}

\begin{figure}
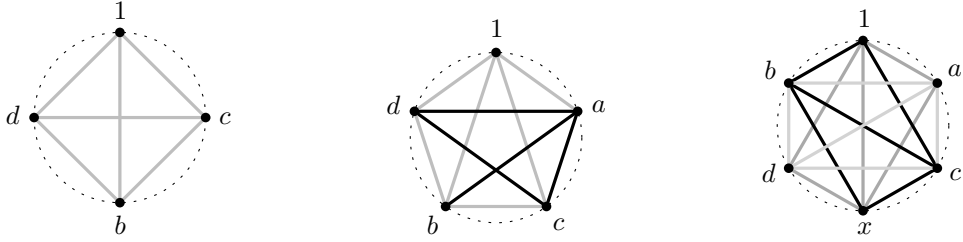

\begin{center}
  \fig6\hf\fig7\hf\fig8
\end{center}
\caption{Colouring of chords in the three cases of the proof of
  Lemma~\ref{lem:critical}. The colouring uses one colour (left), two
  colours (middle) and three colours (right) represented by shades of
  grey.}
\label{fig:chords}
\end{figure}

\bibliographystyle{plain}

\end{document}